\documentclass[11pt,letterpaper]{article}

\usepackage{amsfonts, amsmath, amssymb, amscd, amsthm, graphicx, color}
\usepackage{ifthen}
\usepackage[colorlinks=true, linkcolor=blue, citecolor=magenta, urlcolor=blue]{hyperref}
\usepackage{todo}

\newtheorem{thm}{Theorem}[section]

\newtheorem{lem}[thm]{Lemma}
\newtheorem{prop}[thm]{Proposition}

\theoremstyle{definition}
\newtheorem{defn}[thm]{Definition}
\theoremstyle{remark}
\newtheorem{rem}[thm]{Remark}

\newfont{\eufm}{eufm10}

\renewcommand{\P }{\mathcal P}
\renewcommand{\kappa }{\varkappa}

\renewcommand{\ll }{\langle\hspace{-.7mm}\langle }
\newcommand{\rr }{\rangle\hspace{-.7mm}\rangle }
\renewcommand{\P}{\mathcal P}

\newcommand*{\dist}[3][]{
	\ifthenelse{\equal{#1}{}}
		{\mathchoice%
			{d\!\left(#2,#3\right)}%
			{d(#2,#3)}%
			{d(#2,#3)}%
			{d(#2,#3)}%
		}
		{\mathchoice%
			{d_{#1}\!\left(#2,#3\right)}%
			{d_{#1}(#2,#3)}%
			{d_{#1}(#2,#3)}%
			{d_{#1}(#2,#3)}%
		}
}

\newcommand*{\distV}[1][]{
	\ifthenelse{\equal{#1}{}}
		{d}
		{d_{#1}}
}

\newcommand{\norm}[2][]{%
	\ifthenelse{\equal{#1}{}}%
	{%
		\mathchoice%
		{\left\| #2 \right\|}%
		{\|#2\|}%
		{\|#2\|}%
		{\|#2\|}%
	}%
	{%
		\mathchoice%
		{\left\| #2 \right\|_{#1}}%
		{\|#2\|_{#1}}%
		{\|#2\|_{#1}}%
		{\|#2\|_{#1}}%
	}%
}

\newcommand{\gro}[4][]{
	\ifthenelse{\equal{#1}{}}
	{\mathchoice%
		{\left\langle #2,#3\right\rangle_{#4}}%
		{\langle #2,#3\rangle_{#4}}%
		{\langle #2,#3\rangle_{#4}}%
		{\langle #2,#3\rangle_{#4}}%
	}
	{\mathchoice%
		{\left\langle #2,#3\right\rangle_{#4}^{#1}}%
		{\langle #2,#3\rangle_{#4}^{#1}}%
		{\langle #2,#3\rangle_{#4}^{#1}}%
		{\langle #2,#3\rangle_{#4}^{#1}}%
	}
}

\newcommand{\stab}[1]{%
	\mathchoice%
	{\operatorname{Stab}\left( #1\right)}%
	{\operatorname{Stab}( #1)}%
	{\operatorname{Stab}( #1)}%
	{\operatorname{Stab}( #1)}%
}

\newcommand{\inj}[2][]{
	\ifthenelse{\equal{#1}{}}
		{\operatorname{inj}\left(#2\right)}
		{\operatorname{inj}\left(#2,#1\right)}	
}

\newcommand{\set}[2]{%
	\mathchoice
	{\left\{#1\;\middle|\ #2\right\}}%
	{\{#1\;|\  #2\}}%
	{\{#1\;|\  #2\}}%
	{\{#1\;|\  #2\}}%
}

\newcommand{\diam}{\operatorname{diam}}

\newcommand{\card}[1]{
	\left| #1\right|
}

\newcommand{\lnormal}{%
	\langle\!\langle%
}
\newcommand{\rnormal}{%
	\rangle\!\rangle%
}

\newcommand{\normal}[1]{%
	\lnormal #1 \rnormal
}

\begin{document}

\title{A non-residually finite group acting uniformly properly on a hyperbolic space}
\author{R. Coulon\thanks{The first author acknowledges the support of the ANR grant DAGGER ANR-16-CE40-0006-01. He is also grateful to the \emph{Centre Henri Lebesgue} ANR-11-LABX-0020-01 for creating an attractive mathematical environment.}, D. Osin\thanks{The work of the second author has been supported by the NSF grant DMS-1612473.}}
\date{}

\maketitle

\begin{abstract}
	In this article we produce an example of a non-residually finite group which admits a uniformly proper action on a Gromov hyperbolic space.
\end{abstract}

%%%%%%%%%%%%%%%%%%%%%%%%%%%%%%%%%%%%%%%%%%%%%%%%%%%%%%%%%%%%%%%%%%%%%%%%%%%
%%%%%%%%%%%%%%%%%%%%%%%%%%%%%%%%%%%%%%%%%%%%%%%%%%%%%%%%%%%%%%%%%%%%%%%%%%%
%
\section{Introduction}
%
%%%%%%%%%%%%%%%%%%%%%%%%%%%%%%%%%%%%%%%%%%%%%%%%%%%%%%%%%%%%%%%%%%%%%%%%%%%
%%%%%%%%%%%%%%%%%%%%%%%%%%%%%%%%%%%%%%%%%%%%%%%%%%%%%%%%%%%%%%%%%%%%%%%%%%%

By default, all actions of groups on metric spaces considered in this paper are by isometries.
Recall that a group is \emph{hyperbolic} if and only if it acts properly and cocompactly on a hyperbolic metric space.
It is natural to ask what kind of groups we get if we remove the requirement of cocompactness from this definition. However, it turns out that every countable group admits a proper action on a hyperbolic space, namely the parabolic action on a combinatorial horoball \cite{Groves:2008ip}. Thus to obtain an interesting class of groups we have to strengthen our properness assumptions.

\medskip
In this paper we propose to study the class of groups that admit a uniformly proper action on a hyperbolic length space. We denote this class of group by $\P$. Recall that an action of a group $G$ on a metric space $X$ is \emph{uniformly proper} if for every $r\in \mathbb R_+$, there exists $N\in \mathbb N$, such that for all $x \in X$,
\begin{displaymath}
	\card{\set{g\in G}{\dist[X] x{gx}\leq r}}\leq N.
\end{displaymath}
Having a uniformly proper action on a hyperbolic space is a rather restrictive condition. 
For instance, \cite[Theorem 1.2]{Osi16} implies that every group $G\in \P$ (as well as every its subgroup) is either virtually cyclic or acylindrically hyperbolic, which imposes strong restrictions on the algebraic structure of $G$.

\medskip
Hyperbolic groups and their subgroups obviously belong to $\P$ and, in general, groups from the class $\P$ have many properties similar to those of hyperbolic groups. In fact, we do not know the answer to the following question: \emph{Does $\P$ coincide with the class of all subgroups of hyperbolic groups?} Although the affirmative answer seems unlikely, we are not aware of any obvious counterexamples.

\medskip

This paper is inspired by the well-known open problem of whether every hyperbolic group is residually finite. Our main result shows that the answer to this question is negative if one replaces the class of hyperbolic groups with the class $\P$.

\begin{thm}\label{main}
There exists a finitely generated non-trivial group $G$ acting uniformly properly on a hyperbolic graph of bounded valence such that every amenable quotient of $G$ is trivial. In particular, $G\in \mathcal P$ and $G$ is not residually finite.
\end{thm}

In the process of constructing such a group $G$, we show that a subclass of $\P$ is closed under taking certain small cancellation quotients (see Section~\ref{sec: sc}). This result seems to be of independent interest and can potentially be used to construct other interesting examples of groups from the class $\P$. 

The proof of the second claim of Theorem \ref{main} can be illustrated as follows. 
We first use a variant of the Rips construction suggested in \cite{BO} to construct a subgroup $N$ of a torsion-free hyperbolic group $H$ and two elements $a,b\in N$ which are ``sufficiently independent" in $N$ (more precisely, non-commensurable - see Section~\ref{sec: hyp geom} for the definition) but are conjugate in every finite quotient of $N$. 
The fact that these elements are ``sufficiently independent" together with the result about small cancellation quotients mentioned above imply that the quotient group $G=N/\ll a^p, b^q\rr$ belongs to $\P$ for some (in fact, all sufficiently large) primes $p$ and $q$. 
If $p\ne q$, the images of $a$ and $b$ are clearly trivial in every finite quotient of $G$. 
In particular, $G$ is not residually finite. A slightly more elaborated version of this idea involving Kazhdan's property (T) leads to the proof of the first claim of the theorem.

\paragraph{Acknowledgments.} We are grateful to Ashot Minasyan for useful comments and suggestions, which allowed us to simplify the original proof of Theorem \ref{thm-main}.

%%%%%%%%%%%%%%%%%%%%%%%%%%%%%%%%%%%%%%%%%%%%%%%%%%%%%%%%%%%%%%%%%%%%%%%%%%%
%%%%%%%%%%%%%%%%%%%%%%%%%%%%%%%%%%%%%%%%%%%%%%%%%%%%%%%%%%%%%%%%%%%%%%%%%%%
%
\section{A short review of hyperbolic geometry}
%
%%%%%%%%%%%%%%%%%%%%%%%%%%%%%%%%%%%%%%%%%%%%%%%%%%%%%%%%%%%%%%%%%%%%%%%%%%%
%%%%%%%%%%%%%%%%%%%%%%%%%%%%%%%%%%%%%%%%%%%%%%%%%%%%%%%%%%%%%%%%%%%%%%%%%%%
\label{sec: hyp geom}

In this section we recall a few notations and definitions regarding hyperbolic spaces in the sense of Gromov.
For more details, refer the reader to Gromov's original article \cite{Gro87} or \cite{CooDelPap90,Ghys:1990ki}.

\paragraph{The four point inequality.}
Let $(X,d)$ be a length space.
Recall that the \emph{Gromov product} of three points $x,y,z \in X$  is defined by
\begin{displaymath}
	\gro xyz = \frac 12 \left\{  \dist xz + \dist yz - \dist xy \right\}.
\end{displaymath}
In the remainder of this section, we assume that  $X$ is \emph{$\delta$-hyperbolic}, i.e. for every $x,y,z,t \in X$,
\begin{equation}
	\gro xzt \geq \min\left\{ \gro xyt, \gro yzt \right\} - \delta.
\end{equation}
We denote by $\partial X$ the boundary at infinity of $X$, see \cite[Chapitre 2]{CooDelPap90}.

\paragraph{Quasi-convex subsets.}
Let $Y$ be a subset of $X$.
Recall that $Y$ is \emph{$\alpha$-quasi-convex} if for every $x \in X$, for every $y,y' \in Y$, we have $d(x,Y) \leq \gro y{y'}x + \alpha$.
If $Y$ is path-connected, we denote by $\distV[Y]$ the length pseudo-metric on $Y$ induced by the restriction of $\distV[X]$ on $Y$.
The set $Y$ is \emph{strongly quasi-convex} if $Y$ is $2\delta$-quasi-convex and for every $y,y' \in Y$ we have
\begin{displaymath}
	\dist[X] y{y'} \leq \dist[Y] y{y'} \leq \dist[X] y{y'} + 8\delta.
\end{displaymath}
We denote by $Y^{+\alpha}$, the \emph{$\alpha$-neighborhood} of $Y$, i.e. the set of points $x \in X$ such that $d(x, Y) \leq \alpha$.

\paragraph{Group action.}
Let $G$ be a group acting uniformly properly on $X$.
An element $g \in G$ is either \emph{elliptic} (it has bounded orbits, hence finite order) or \emph{loxodromic} (it has exactly two accumulation points in $\partial X$) \cite[Lemma~2.2]{Bowditch:2008bj}.
A subgroup of $G$ is either \emph{elementary} (it is virtually cyclic) or contains a copy of the free group $\mathbb F_2$ \cite[Paragraph~8.2]{Gro87}.
In order to measure the action of $g$ on $X$, we use the translation length defined as follows
\begin{equation*}
	\norm[X] g = \inf_{x \in X} \dist {gx}x.
\end{equation*}
If there is no ambiguity, we omit the space $X$ in the notation.
A loxodromic element $g \in G$ fixes exactly two points $g^-$ and $g^+$ in $\partial X$.
We denote by $E(g)$ the stabilizer of $\left\{ g^-, g^+\right\}$.
It is the maximal elementary subgroup containing $g$.
Moreover $\langle g \rangle$ has finite index in $E(g)$ \cite[Lemma 6.5]{Dahmani:2017ef}.

\medskip
Given a loxodromic element $g \in G$, there exists a $g$-invariant strongly quasi-convex subset $Y_g$ of $X$ which is quasi-isometric to line; its stabilizer is $E(g)$ and the quotient $Y_g/E(g)$ is bounded \cite[Definition~3.12 and Lemma~3.13]{Coulon:2016if}.
We call this set $Y_g$ the \emph{cylinder} of $g$.

\medskip
We say that two elements $g,h \in G$ are \emph{commensurable}, if there exists $n,m \in \mathbb Z^*$ and $u \in G$ such that $g^n = uh^mu^{-1}$.
Every loxodromic element is contained in a unique maximal elementary subgroup \cite[Lemma~3.28]{Coulon:2016if}.
Hence two loxodromic elements $g$ and $h$ are commensurable if and only if there exits $u \in G$ such that $g$ and $uhu^{-1}$ generate an elementary subgroup.

\begin{lem}
\label{res: fellow travel conjugate}
	Let $S$ be a finite collection of pairwise non commensurable loxodromic elements of $G$.
	There exists $\Delta \in \mathbb R_+$ with the following property.
	For every $g,g' \in S$, for every $u \in G$, if
	\begin{displaymath}
		\diam \left(Y_g^{+5\delta} \cap uY_{g'}^{+5\delta}\right) > \Delta,
	\end{displaymath}
	then $g = g'$ and $u \in E(g)$.
\end{lem}

\begin{proof}
	The action of $G$ on $X$ being uniformly proper, it is also acylindrical.
	According to \cite[Proposition~3.44 and Lemma~6.14]{Coulon:2016if} there exists a constant $A,B >0$ with the following property:
	if $h,h' \in G$ are two loxodromic elements generating a non-elementary subgroup, then
	\begin{displaymath}
		\diam \left(Y_h^{+5\delta} \cap Y_{h'}^{+5\delta}\right) \leq A \max\{ \norm h, \norm {h'}\} + B.
	\end{displaymath}
	We now let
	\begin{displaymath}
		\Delta = A\max_{g \in S} \norm g + B.
	\end{displaymath}
	Let $g,g' \in S$, and $u \in G$ such that
	\begin{displaymath}
		\diam \left(Y_g^{+5\delta} \cap uY_{g'}^{+5\delta}\right) > \Delta.
	\end{displaymath}
	Recall that $uY_{g'}$ is the cylinder of $ug'u^{-1}$.
	It follows from our choice of $\Delta$, that $g$ and $ug'u^{-1}$ generate an elementary subgroup.
	Since the elements of $S$ are pairwise non-commensurable it forces $g = g'$ and $u \in E(g)$.
\end{proof}

%%%%%%%%%%%%%%%%%%%%%%%%%%%%%%%%%%%%%%%%%%%%%%%%%%%%%%%%%%%%%%%%%%%%%%%%%%%
%%%%%%%%%%%%%%%%%%%%%%%%%%%%%%%%%%%%%%%%%%%%%%%%%%%%%%%%%%%%%%%%%%%%%%%%%%%
%
\section{An auxiliary class $\P_0$}
%
%%%%%%%%%%%%%%%%%%%%%%%%%%%%%%%%%%%%%%%%%%%%%%%%%%%%%%%%%%%%%%%%%%%%%%%%%%%
%%%%%%%%%%%%%%%%%%%%%%%%%%%%%%%%%%%%%%%%%%%%%%%%%%%%%%%%%%%%%%%%%%%%%%%%%%%

To prove our main result we will make use of an auxiliary class $\mathcal P_0$.

\begin{defn}
\label{def: bounded geometry}
	A subset $S$ of a metric space $X$ is \emph{$r$-separated} if for every distinct points $s,s' \in S$, $\dist s{s'} \geq r$.
	Given a subset $Y$ of $X$ and $r >0$, we define the \emph{$r$-capacity} of $Y$, denoted by $C_r(Y)$, as the maximal number of points in an $r$-separated subset of $Y$.
	We say that $X$ has \emph{$r$-bounded geometry} if for every $R>0$, there is  an integer $N$ bounding from above  the $r$-capacity of every ball of radius $R$.
	If there exists $r>0$ such that $X$ has $r$-bounded geometry we simply say that $X$ has \emph{bounded geometry}.
\end{defn}

The class $\P_0$ we are interested in consists of all groups admitting a uniformly proper action on a hyperbolic length space with bounded geometry.
It is clear that $\P_0\subseteq \P$.
We will show that the class $\P_0$ is closed under certain small cancellation quotients.
Before we discuss the precise statements and proofs, a few remarks are in order.
First, we do not know whether $\P_0$ is indeed a proper subclass of $\P$.
Second, it is possible to prove the results of the next section for the whole class $\P$.
Nevertheless, the proofs become much easier for $\P_0$.
Therefore we restrict our attention to this subclass, which is sufficient for the proof of our main theorem.

\medskip
We start with a few equivalent characterizations of the class $\mathcal P_0$.
A graph  $\Gamma = (V,E)$ is understood here in the sense of Serre \cite{Serre:1977wy}.
Observe that a graph $\Gamma$ has bounded geometry whenever it has \emph{uniformly bounded valence} i.e. there exists $d \in \mathbb N$, such that the valence of any vertex $v \in V$ is at most $d$.

\begin{rem}
The converse statement is false. Indeed, consider the real line, which we think of as a graph with the vertex set $\mathbb Z$ and the obvious edges; to each vertex, attach infinitely many edges of length $1$. The resulting graph has $3$-bounded geometry while some vertices have infinite valence.
\end{rem}

If $\Gamma$ is a graph with uniformly bounded valence, the action of a group $G$ on $\Gamma$ is uniformly proper if and only if there exists $N \in \mathbb N$ such that the stabilizer of any vertex contains at most $N$ elements.

\begin{prop}
\label{res: characterization P0}
	Let $G$ be a group.
	The following assertions are equivalent.
	\begin{enumerate}
		\item \label{enu: characterization P0 - def}
		$G$ belongs to $\mathcal P_0$.
		\item  \label{enu: characterization P0 - action w/o inversion}
		$G$ acts uniformly properly without inversion on a hyperbolic graph $\Gamma$ with uniformly bounded valence.
		\item \label{enu: characterization P0 - free action}
		$G$ acts on a hyperbolic graph $\Gamma$ with uniformly bounded valence such that the action of $G$ is free when restricted to the vertex set of $\Gamma$.
		
	\end{enumerate}
\end{prop}

\begin{proof}
	To show that (\ref{enu: characterization P0 - free action}) $\Rightarrow$ (\ref{enu: characterization P0 - action w/o inversion}) one simply takes the barycentric subdivision of the graph.
	The implication (\ref{enu: characterization P0 - action w/o inversion}) $\Rightarrow$ (\ref{enu: characterization P0 - def}) directly follows from the definition.
	We now focus on (\ref{enu: characterization P0 - def}) $\Rightarrow$ (\ref{enu: characterization P0 - free action}).
	By definition there exists $r \in \mathbb R_+^*$ such that $G$ acts uniformly properly on a hyperbolic length space $X$ with $r$-bounded geometry.
	Using Zorn's Lemma we choose an $r$-separated subset $\bar S$ of $\bar X = X/G$ which is maximal for this property.
	We denote by $S$ the pre-image of $\bar S$ in $X$.
	We fix $S_0 \subset S$ to be a set of representatives for the action of $G$ on $S$.
 	Let $R = 2r + 1$.
	We now define a graph $\Gamma = (V,E)$ as follows.
	Its vertex set is $V = G \times S_0$.
	The edge set $E$ is the set of pairs $((u,s),(u',s')) \in V \times V$ such that $\dist[X]{us}{u's'} \leq R$.
	The initial and terminal vertices of such an edge are $(u,s)$ and $(u',s')$ respectively.
	The group $G$ acts freely on $V$ as follows: for every $g \in G$, for every $(u,s) \in V$, we have $g \cdot (u,s) = (gu,s)$.
	This action induces an action by isometries of $G$ on $\Gamma$.
	Recall that $R > 2r$.
	This allows us to perform a variation on the Milnor-Svar\v c Lemma and prove that the map $V \to X$ sending $(u,s)$ to $us$ induces a ($G$-equivariant) quasi-isometry from $\Gamma$ to $X$.
	In particular $\Gamma$ is hyperbolic.
	We are left to prove that $\Gamma$ has uniformly bounded valence.
	
	\medskip
	Since $X$ has $r$-bounded geometry, there exists $N_1 \in \mathbb N$ such that the $r$-capacity of any ball of radius $R$ in $X$ is at most $N_1$.
	The group $G$ acting uniformly properly on $X$, there exists $N_2 \in \mathbb N$ such that for every $x \in X$, the cardinality of the set
	\begin{displaymath}
		U(x) = \left\{ g \in G \mid \dist[X] x{gx}\leq 2R \right\}
	\end{displaymath}
	is bounded above by $N_2$.
	We now fix a vertex $v_0 = (u_0,s_0)$ of $\Gamma$.
	We fix a subset $S_1$ of $B(u_0s_0,R)$ such that any $G$-orbit of $S$ intersecting $B(u_0s_0,R)$ contains exactly one point in $S_1$.
	It follows from our choice of $S$ that if $s, s' \in S$ belong to distinct $G$-orbits, then $\dist[X]s{s'} \geq r$.
	Consequently the cardinality of $S_1$ is bounded above by the $r$-capacity of this ball, i.e. $N_1$.
	By construction for every $s \in S_1$, there exists $u_s \in G$ such that $u_ss$ belongs to $S_0$.
	It follows from the definition of $\Gamma$ combined with the triangle inequality that any neighbor of $v_0$ belongs to the set
	\begin{displaymath}
		\left\{ (uu_s^{-1},u_ss) \mid s \in S_1, u \in U(s) \right\}.
	\end{displaymath}
	The cardinality of this set is bounded above by $d = N_1N_2$, which does not depend on $v_0$, hence $\Gamma$ has uniformly bounded valence.
	\end{proof}

%%%%%%%%%%%%%%%%%%%%%%%%%%%%%%%%%%%%%%%%%%%%%%%%%%%%%%%%%%%%%%%%%%%%%%%%%%%
%%%%%%%%%%%%%%%%%%%%%%%%%%%%%%%%%%%%%%%%%%%%%%%%%%%%%%%%%%%%%%%%%%%%%%%%%%%
%
\section{Stability of the class $\P_0$.}
%
%%%%%%%%%%%%%%%%%%%%%%%%%%%%%%%%%%%%%%%%%%%%%%%%%%%%%%%%%%%%%%%%%%%%%%%%%%%
%%%%%%%%%%%%%%%%%%%%%%%%%%%%%%%%%%%%%%%%%%%%%%%%%%%%%%%%%%%%%%%%%%%%%%%%%%%
\label{sec: sc}

We now explain how $\mathcal P_0$ behaves under small cancellation.
To that end we first review the geometric theory of small cancellation as it has been introduced by M.~Gromov \cite{Gro01b} and further developed in \cite{DelGro08,Coulon:il,Dahmani:2017ef}.
For a detailed exposition we refer the reader to \cite[Sections~4-6]{Coulon:2014fr}.

\paragraph{Settings.}
Let $X$ be a $\delta$-hyperbolic length space and $G$ a group acting on $X$.
Let $\mathcal Q$ be a family of pairs $(H,Y)$ such that $Y$ is a strongly quasi-convex subset of $X$ and $H$ a subgroup of $\stab Y$.
We assume that $\mathcal Q$ is closed under the following action of $G$:
for every $(H,Y) \in \mathcal Q$, for every $g \in G$, $g(H,Y) = (gHg^{-1},gY)$.
In addition we require that $\mathcal Q/G$ is finite.
We denote by $K$ the (normal) subgroup generated by the subgroups $H$ where $(H,Y) \in \mathcal Q$.
The goal is to study the quotient $\bar G = G/K$ and the corresponding projection $\pi \colon G \to \bar G$.
To that end we define the following two small cancellation parameters
\begin{eqnarray*}
	\Delta (\mathcal Q,X) & = & \sup \set{\diam\left(Y_1^{+5\delta}\cap Y_2^{+5\delta}\right)}{ (H_1,Y_1) \neq (H_2,Y_2) \in \mathcal Q}, \\
	\inj[X]{\mathcal Q} & = & \inf \set{\norm h}{h \in H\setminus\{1\},\; (H,Y) \in \mathcal Q}.
\end{eqnarray*}
They play the role of the length of the longest piece and the shortest relation respectively.
We now fix a number $\rho>0$.
Its value will be made precise later.
It should be thought of as a very large parameter.

\paragraph{Cones.}
Let $(H,Y) \in \mathcal Q$.
The \emph{cone of radius $\rho$ over $Y$}, denoted by $Z(Y)$, is the quotient of $Y\times [0,\rho]$ by the equivalence relation that identifies all the points of the form $(y,0)$.
The equivalence class of $(y,0)$, denoted by $a$, is called the \emph{apex} of the cone.
By abuse of notation, we still write $(y,r)$ for the equivalence class of $(y,r)$.
The map $\iota \colon Y \rightarrow Z(Y)$ that send $y$ to $(y,\rho)$ provides a natural embedding form $Y$ to $Z(Y)$.
This space can be endowed with a metric as described below.
For the geometric interpretation of the distance see \cite[Section~4.1]{Coulon:2014fr}.
	
\begin{prop}{\rm \cite[Chapter I.5, Proposition 5.9]{BriHae99}} \quad
\label{res: def distance cone}
	The cone $Z(Y)$ is endowed with a metric characterized in the following way.
	Let $x=(y,r)$ and $x'=(y',r')$ be two points of $Z(Y)$ then
	\begin{displaymath}
		\cosh \dist[Z(Y)] x{x'} = \cosh r \cosh r' - \sinh r\sinh r' \cos \theta(y,y'),
	\end{displaymath}
	where $\theta(y,y')$ is the \emph{angle at the apex} defined by $\theta(y,y') = \min \left\{ \pi , {\dist[Y]y{y'}}/\sinh \rho\right\}$.
\end{prop}

\paragraph{Cone-off over a metric space.}
The \textit{cone-off of radius $\rho$ over $X$ relative to $\mathcal Y$} denoted by $\dot X_\rho(\mathcal Q)$ (or simply $\dot X$) is obtained by attaching for every $(H,Y) \in \mathcal Q$, the cone $Z(Y)$ on $X$ along $Y$ according to $\iota$.
We endow $\dot X$ with the largest pseudo-metric $\distV[\dot X]$ for which all the maps $X \to \dot X$ and $Z(Y) \to \dot X$ -- when $(H,Y)$ runs over $\mathcal Q$ -- are $1$-Lipschitz.
It turns out that this pseudo-distance is actually a distance on $\dot X$ \cite[Proposition 5.10]{Coulon:2014fr}.
Actually $(\dot X, \distV[\dot X])$ is a length space.

\medskip
The action of $G$ on $X$ naturally extends to an action by isometries on $\dot X$ as follows.
Let $(H,Y) \in \mathcal Q$.
For every $x = (y,r) \in Z(Y)$, for every $g \in G$, $gx$ is the point of $Z(gY)$ defined by $gx = (gy,r)$.
The space $\bar X_\rho(\mathcal Q)$ (or simply $\bar X$) is the quotient $\bar X = \dot X/K$.
The metric on $\dot X$ induces a pseudo-metric on $\bar X$.
We write $\zeta \colon \dot X \rightarrow \bar X$ for the canonical projection from $\dot X$ to $\bar X$.
The quotient $\bar G$ naturally acts by isometries on $\bar X$.

\begin{prop}
\label{res: qi quotient spaces}
	Assume that for every $(H,Y) \in \mathcal Q$, the space $Y/H$ is bounded.
	Then the spaces $\bar X$ and $X/K$ are quasi-isometric.
\end{prop}

\begin{proof}
	Recall that the embedding $X \to \dot X$ is $1$-Lipschitz.
	Hence it induces a $1$-Lipschitz embedding $X/K \to \bar X$.
	We claim that the map $X/K \to \bar X$ is actually bi-Lipschitz.
	For simplicity, we implicitly identify $X/K$ with its image in $\bar X$.
	Recall that $\mathcal Q/G$ is finite.
	It follows from our assumption that there exists $D \in \mathbb R_+$ such that for every $(H,Y) \in \mathcal Q$, the image of $Y$ in $X/K$ has diameter at most $D$.
	
	\medskip
	Let $\bar x, \bar x' \in X/K$.
	Let $\eta \in \mathbb R_+^*$.
	There exist $x,x' \in X$, respective pre-images of $\bar x$ and $\bar x'$, such that $\dist[\dot X] x{x'} < \dist[\bar X] {\bar x}{\bar x'} + \eta$.
	Following the construction of the metric on $\dot X$ -- see for instance \cite[Section~5.1]{Coulon:2014fr} -- we observe that there exists a sequence of points $(x_0,y_0,x_1, y_1, \dots, x_m,y_m)$ which approximates the distance between $x$ and $x'$ in the following sense:
	\begin{enumerate}
		\item $x_0  =x$ and $y_m = x'$;
		\item For every $i \in \{ 0, \dots, m-1\}$, there exists $(H_i,Y_i) \in \mathcal Q$ such that $y_i,x_{i+1} \in Y_i$;
		\item
		\begin{equation}
		\label{eqn: qi quotient spaces - approx dot X}
			\sum_{i = 0}^m \dist[X] {x_i}{y_i} + \sum_{i = 0}^{m-1} \dist[Z(Y_i)]{y_i}{x_{i+1}} < \dist[\dot X] x{x'} + \eta.
		\end{equation}
	\end{enumerate}
	For every $i \in \{0, \dots, m\}$, we write $\bar x_i$ and $\bar y_i$ for the images in $X/K$ of $x_i$ and $y_i$ respectively.
	It follows from the triangle inequality that
	\begin{equation}
	\label{eqn: qi quotient spaces - approx bar X}
		\dist[X/K]{\bar x}{\bar x'}
		\leq \sum_{i = 0}^m \dist[X/K] {\bar x_i}{\bar y_i} + \sum_{i = 0}^{m-1} \dist[X/K]{\bar y_i}{\bar x_{i+1}}
	\end{equation}
	We are going to compare the terms of the latter inequality with the ones of (\ref{eqn: qi quotient spaces - approx dot X}).
	Note first that for every $i \in \{0, \dots, m\}$, we have
	\begin{equation}
	\label{eqn: qi quotient spaces - easy comparison}
		\dist[X/K] {\bar x_i}{\bar y_i} \leq \dist[X] {x_i}{y_i}.
	\end{equation}
	Let $i \in \{0, \dots, m-1\}$.
	In order to estimate $\dist[X/K]{\bar y_i}{\bar x_{i+1}}$, we distinguish two cases.
	Assume first that $\dist[Y_i]{y_i}{x_{i+1}} \leq \pi \sinh \rho$.
	It follows from the definition of the metric on $Z(Y_i)$ that
	\begin{equation}
	\label{eqn: qi quotient spaces - hard comparison - case 1}
		\dist[X/K]{\bar y_i}{\bar x_{i+1}}
		\leq \dist[X]{y_i}{x_{i+1}}
		\leq \dist[Y_i]{y_i}{x_{i+1}}
		\leq \frac {\pi \sinh \rho}{2\rho} \dist[Z(Y_i)]{y_i}{x_{i+1}}.
	\end{equation}
	Assume now that $\dist[Y_i]{y_i}{x_{i+1}} > \pi \sinh \rho$.
	In particular $\dist[Z(Y_i)]{y_i}{x_{i+1}}  = 2 \rho$.
	Recall that the diameter of the image of $Y_i$ in $X/K$ is at most $D$.
	Hence
	\begin{equation}
	\label{eqn: qi quotient spaces - hard comparison - case 2}
		\dist[X/K]{\bar y_i}{\bar x_{i+1}} \leq \frac D{2\rho} \dist[Z(Y_i)]{y_i}{x_{i+1}}.
	\end{equation}
	Combining (\ref{eqn: qi quotient spaces - approx dot X}) - (\ref{eqn: qi quotient spaces - hard comparison - case 2}) we get that
	\begin{displaymath}
		\dist[X/K]{\bar x}{\bar x'}
		\leq \lambda \left(\dist[\dot X]x{x'} + \eta\right)
		\leq \lambda \left(\dist[\bar X]{\bar x}{\bar x'} + 2 \eta \right),
	\end{displaymath}
	where
	\begin{displaymath}
		\lambda = \max \left\{ 1, \frac {\pi \sinh\rho}{2\rho}, \frac D{2\rho} \right\}.
	\end{displaymath}
	The previous inequality holds for every $\eta \in \mathbb R_+^*$, hence $X/K \to \bar X$ is bi-Lipschitz, which completes the proof of our claim.
	Note that the diameter of the cones attached to $X$ to form the cone-off space $\dot X$ have diameter at most $2\rho$.
	Hence any point of $\bar X$ is a distance at most $2\rho$ from a point of $X/K$.
	Consequently the map $X/K \to \bar X$ is a quasi-isometry.
\end{proof}

\paragraph{Small cancellation theorem.}
The small cancellation theorem recalled bellow is a compilation of Proposition~6.7, Corollary~3.12, and Proposition~6.12 from \cite{Coulon:2014fr}

\begin{thm}
\label{res: small cancellation theorem}
	There exist positive constants $\delta_0, \delta_1$, $\Delta_0$ and $\rho_0$ satisfying the following property.
	Let $X$ be a $\delta$-hyperbolic length space and $G$ a group acting by isometries on $X$.
	Let $\mathcal Q$ be a $G$-invariant family of pairs $(H,Y)$ where $Y$ is a strongly quasi-convex subset of $X$ and $H$ a subgroup of $G$ stabilizing $Y$.
	We assume that $\mathcal Q/G$ is finite.
	Let $\rho \geq \rho_0$.
	If $\delta \leq \delta_0$, $\Delta(\mathcal Q,X) \leq \Delta_0$ and $\inj[X]{\mathcal Q} \geq 2 \pi \sinh \rho$ then the following holds.
	\begin{enumerate}
		\item The space $\bar X  = \bar X_\rho(\mathcal Q)$ is a $\delta_1$-hyperbolic length space.
		\item Let $(H,Y) \in \mathcal Q$.
		Let $a$ be the apex of $Z(Y)$ and $\bar a$ its image in $\bar X$.
		The projection $\pi \colon G \twoheadrightarrow \bar G$ induces an isomorphism from $\stab Y/H$ onto $\stab{\bar a}$.
		\item For every $x \in X$, the projection $\pi \colon G \to \bar G$ induces a bijection from the set $\set{g \in G}{\dist{gx}x \leq \rho/100}$ onto its image.
		\item Let $\bar F$ be an elliptic subgroup of $\bar G$.
		Either there exists an elliptic subgroup $F$ of $G$ such that the projection $\pi \colon G \to \bar G$ induces an isomorphism from $F$ onto $\bar F$, or there exists $(H,Y) \in \mathcal Q$ such that $\bar F$ is contained in $\stab{\bar a}$, where $\bar a$ stands for the image in $\bar X$ of the apex $a$ of the cone $Z(Y)$.
	\end{enumerate}
\end{thm}

We are now in position to prove the following statement.

\begin{prop}
\label{SCQ}
Let $G$ be a group acting uniformly properly without inversion on a hyperbolic graph $\Gamma$ with uniformly bounded valence.
Let $\{g_1, \dots, g_m\}$ be a finite subset of $G$ whose elements are loxodromic (with respect to the action of $G$ on $\Gamma$) and pairwise non-commensurable.
In addition, we assume that for every $i\in\{1, \dots, m\}$, the group $\langle g_i \rangle$ is normal in $E(g_i)$.
Then for every finite subset $U\subseteq G$ there exists $N \in \mathbb N$ with the following property.
Let $n_1, \dots, n_m \in \mathbb N$, all bounded below by $N$.
Let $K$ be the normal closure of $\{g_1^{n_1}, \dots, g_m^{n_m}\}$ in $G$.
Then the quotient $\bar G = G/K$ belongs to $\P_0$. Moreover, we have the following.
	\begin{enumerate}
		\item \label{enu: sc quotient w/ bounded geometry - cone point stab}
		For every $i\in\{1, \dots, m\}$, the natural homomorphism $\pi \colon G \to \bar G$ induces an embedding of $E(g_i) / \langle g_i \rangle$ into $\bar G$.
		\item \label{enu: sc quotient w/ bounded geometry - local one-to-one}
		 The projection $\pi$ is injective when restricted to $U$.
		\item \label{enu: sc quotient w/ bounded geometry - elliptic}
		Let $\bar F$ be a finite subgroup of $\bar G$. Then either there exists a finite subgroup $F$ of $G$ such that $\pi (F)=\bar F$ or $\bar F$ is conjugate to a subgroup of $\pi(E(g_i))$ for some $i\in\{1, \dots, m\}$.
	\end{enumerate}
\end{prop}

\begin{proof}
	The constant $\delta_0$ $\delta_1$, $\Delta_0$, and $\rho_0$ are the one given by Theorem~\ref{res: small cancellation theorem}.
	We choose an arbitrary $\rho \geq \rho_0$.
	We write $\delta$ for the hyperbolicity constant of $\Gamma$.
	According to Lemma~\ref{res: fellow travel conjugate} there exists a constant $\Delta$ such that for every $u \in G$, for every $i \neq j$ in $\{1, \dots, m\}$, if
	\begin{displaymath}
		\diam\left(Y_{g_i}^{+5\delta}\cap uY_{g_j}^{+5\delta}\right) > \Delta,
	\end{displaymath}
	then $i = j$ and $u$ belongs to $E(g_i)$.
	Up to replacing $\Gamma$ by a rescaled version of $\Gamma$, that we denote $X$, we may assume that the following holds
	\begin{itemize}
		\item $\delta \leq \delta_0$ and $\Delta \leq \Delta_0$,
		\item there exists $x \in X$, such that for every $u \in U$ we have $\dist[X]{ux}x \leq \rho /100$.
	\end{itemize}
	Since the $g_i$'s are loxodromic, there exists $N \in \mathbb N$ such that for every $n \geq N$, for every $i \in \{1, \dots, m\}$, we have $\norm[X]{g_i^n} \geq 2\pi \sinh \rho$.
	Let $n_1, \dots, n_m \in \mathbb N$, all bounded below by $N$.
	Let $K$ be the normal closure of $\{g_1^{n_1}, \dots, g_m^{n_m}\}$ and $\bar G$ be the quotient $\bar G = G/K$.
	
	\medskip
	Since $G$ acts without inversion on $\Gamma$, the quotient of $\bar \Gamma = \Gamma /K$ is a graph endowed with an action without inversion of $\bar G$.
	According to our assumptions there exist $d, M \in \mathbb N$ such that given any vertex $v$ of $\Gamma$, its valence is at most $d$ and the cardinality of its stabilizer is bounded above by $M$.
	Observe that the same holds for the vertices of $\bar \Gamma$.
	To prove that $\bar G$ belongs to $\mathcal P_0$, it suffices to show that $\bar \Gamma$ is hyperbolic.
	To that end, we use small cancellation theory.
	Let $\mathcal Q$ be the following collection
	\begin{equation*}
	\mathcal Q = \set{\left(\langle ug_i^{n_i}u^{-1}\rangle , uY_g\right)}{u \in G,\ 1 \leq i \leq m}.
	\end{equation*}
	By construction $\Delta(\mathcal Q,X) \leq \Delta_0$ and $\inj[X]{\mathcal Q} \geq 2\pi \sinh \rho$.
	The cone-off space $\dot X = \dot X_\rho(\mathcal Q)$ and the quotient $\bar X = \dot X/K$ are built as above.
	The parameters have been chosen in such a way so that the family $\mathcal Q$ satisfies the assumptions of Theorem~\ref{res: small cancellation theorem}.
	It follows that $\bar X$ is a hyperbolic length space.
	Note that for every $(H,Y) \in \mathcal Q$, the quotient $Y/H$ is bounded, hence $\bar X$ is quasi-isometric to $X/K$ (Proposition~\ref{res: qi quotient spaces}).
	Nevertheless $X/K$ is just a rescaled copy of $\bar \Gamma$.
	Thus $\bar \Gamma$ is quasi-isometric to $\bar X$, and therefore hyperbolic.
	Points~(\ref{enu: sc quotient w/ bounded geometry - cone point stab})-(\ref{enu: sc quotient w/ bounded geometry - elliptic}) directly follows from Theorem~\ref{res: small cancellation theorem}.
\end{proof}

\section{Proof of the main theorem}

We begin with an auxiliary result, which is similar to \cite[Proposition 4.2]{MM}.
\begin{lem}
\label{res: rips}
Let $Q$ be a finitely presented infinite simple group and let $H$ be a torsion-free hyperbolic group splitting as
\begin{displaymath}
	1\to N\to H\to Q\to 1,
\end{displaymath}
where the subgroup $N$ is finitely generated. Let $a\in N\setminus \{ 1\}$. Then there exists $b \in N\setminus\{1\}$ such that $a$ and $b$ are not commensurable in $N$ but are conjugate in every finite quotient of $N$.
\end{lem}

\begin{proof}
Let $C=\langle c\rangle $ be the maximal cyclic subgroup of $H$ containing $a$ and let $h\in H\setminus CN$ (note that our assumptions imply that $CN\ne H$). Let $b=h^{-1}ah$ and $a=c^n$ for some $n\in \mathbb Z\setminus \{ 0\}$.

If $a$ and $b$ are commensurable in $N$, then there exist $t\in N$ and $k, \ell\in \mathbb Z\setminus \{ 0\}$ such that $c^{kn}=t^{-1}h^{-1}c^{\ell n}ht$. 
Since $H$ is torsion-free we have $k=\ell$ and by the uniqueness of roots in a torsion-free hyperbolic group we obtain $c=t^{-1}h^{-1}cht$. 
It follows that $ht\in C$ and consequently $h\in CN$, which contradicts our assumption. 
Thus $a$ and $b$ are not commensurable in $N$. 

Assume now that there exists a finite index normal subgroup $K$ of $N$ such that the images of $a$ and $b$ are not conjugate in $N/K$. 
Since $N$ is finitely generated, there are only finitely many subgroups of any finite index in $N$. 
Replacing $K$ with the intersection of all subgroups of $N$ of index $[N:K]$ if necessary, we can assume that $K$ is normal in $H$. 
The natural action of the group $H$ on the finite set $\Omega$ of conjugacy classes of $N/K$ is non-trivial; indeed, the element $h$ acts non-trivially as the images of $a$ and $b$ are not conjugate in $N/K$. 
Since every element of $N$ acts on $\Omega $ trivially, the action of $H$ on $\Omega$ gives rise to a non-trivial homomorphism $\epsilon\colon Q\to Sym(\Omega)$, which contradicts the assumption that $Q$ is infinite simple.
\end{proof}

\begin{thm}\label{thm-main}
There exists a finitely generated group $G\in \P_0$ such that every amenable quotient of $G$ is trivial.
\end{thm}

\begin{proof}
Let $H_1$ be a torsion-free hyperbolic group with property (T) of Kazhdan and 
\begin{equation*}
	H_2=\left\langle x,y\mid y=x\left(y^{-1}xy\right)x^2\left(y^{-1}xy\right) \cdots x^{10} \left(y^{-1}xy\right)\right\rangle.
\end{equation*}
It is easy to see that $H_2$ satisfies the $C^\prime (1/6)$ small cancellation condition and hence is hyperbolic. 
Moreover it is generated by some conjugates of $x$.
Any two non-cyclic torsion-free hyperbolic groups have a common non-cyclic torsion-free hyperbolic quotient group \cite[Theorem~2]{Ols93}. 
Let $H_0$ denote a common non-cyclic torsion-free hyperbolic quotient of $H_1$ and $H_2$.

By \cite[Corollary 1.2]{BO}, there exists a short exact sequence
\begin{displaymath}
	1\to N\to H\to Q\to 1
\end{displaymath}
such that $H$ is torsion-free hyperbolic, $N$ is a quotient of $H_0$, and $Q$ is a finitely presented infinite simple group. 
Clearly $N$ inherits property (T) from $H_1$. 
As a subgroup of a hyperbolic group, $N$ belongs to the class $\mathcal P_0$.
Let $a$ denote the image of $x\in H_2$ in $N$. 
Since $N$ is a quotient group of $H_2$, it is generated by conjugates of $a$ (in $N$). 
%We also note that $N\in \mathcal P_0$ as it acts freely on the Cayley graph of $H$ with respect to a finite generating set.

According to the previous lemma, there exists $b\in N$ such that $a$ and $b$ are not commensurable in $N$ but are conjugate in every finite quotient of $N$. 
By Proposition \ref{SCQ}, there exist distinct primes $p$ and $q$ such that $G=N/\normal{a^p, b^q}$ belongs to $\mathcal P_0$, and the images of $a$ and $b$ in $G$ have orders $p$ and $q$, respectively.

Let $A$ be an amenable quotient of $G$. 
Being a quotient group of $N$, $A$ has property (T) and, therefore, is finite. 
It follows that the images of $a$ and $b$ in $A$, denoted by $\bar a$ and $\bar b$, are conjugate. 
As $\bar a^p=\bar b ^q=1$ and $gcd(p,q)=1$, we have $\bar a= \bar b =1$. 
Since $N$ is generated by conjugates of $a$, $A$ is generated by conjugates of $\bar a$, which implies $A=\{ 1\}$.
\end{proof}

\noindent
\emph{R\'emi Coulon} \\
Univ Rennes, CNRS \\
IRMAR - UMR 6625 \\
 F-35000 Rennes, France\\
\texttt{remi.coulon@univ-rennes1.fr} \\
\texttt{http://rcoulon.perso.math.cnrs.fr}

\bigskip

\noindent
\emph{Denis Osin} \\
Department of Mathematics \\
Vanderbilt University \\
Nashville, TN 37240, U.S.A.\\
\texttt{denis.v.osin@vanderbilt.edu} \\
\texttt{https://as.vanderbilt.edu/math/bio/denis-osin}

\end{document}